\documentclass[12pt]{article}
\usepackage{graphicx}
\usepackage{amsmath, amssymb, amsthm}
\usepackage{mathrsfs}
\usepackage{geometry}
\usepackage{cite}
\usepackage[colorlinks=true]{hyperref}
\usepackage{authblk}
\usepackage{setspace}

\title{The $K_+$-fixed vectors of Iwahori-spherical $GL_n$-representations: connections with Zelevinsky's segments}
\author{Runze Wang}
\date{Peking University, Beijing, China; e-mail: runze\_wang@stu.pku.edu.cn}

\newtheorem{theorem}{Theorem}[section]
\newtheorem{lemma}[theorem]{Lemma}
\newtheorem{corollary}[theorem]{Corollary}
\newtheorem{definition}[theorem]{Definition}

\begin{document}

\maketitle

\begin{abstract}
Let $G = GL_n(F)$ with $F$ a non-archimedean local field of characteristic $0$, and let $K$ be a hyperspecial subgroup of $G$ with pro-unipotent radical $K_+$. Denote by $I$ the standard Iwahori subgroup of $G$. In this paper we study the decomposition of the $K_+$-fixed subspace $V^{K_+}$ into irreducible $K/K_+$-modules, where $V$ is an Iwahori-spherical representation of $G$ (i.e., $V^I \neq 0$).

For a generic (not necessarily irreducible) Iwahori-spherical representation $V$ of the form $V = St(\langle\Delta_1\rangle) \times \cdots \times St(\langle\Delta_r\rangle)$, where each $\Delta_i$ is a segment consisting of unramified characters, we prove that there exists a partition $\lambda = (n_1,\dots,n_r)'$ (the conjugate of the lengths of the segments) such that an irreducible $K/K_+$-module $\pi_\mu$ appears in $V^{K_+}$ if and only if $\mu \trianglelefteq \lambda$ (dominance order). Moreover, the multiplicity of $\pi_\mu$ is given by the Kostka number $K_{\mu'\lambda'}$.

For an arbitrary irreducible Iwahori-spherical representation $St(\langle a\rangle)$ (Zelevinsky dual of $\langle a\rangle$) associated to a multiset $a$ of segments, we associate a partition $P(a)$ defined by the conjugate of the lengths of the segments. We show that if $\pi_\mu$ occurs in $St(\langle a\rangle)^{K_+}$, then necessarily $\mu \trianglelefteq P(a)$. And the multiplicity of $\pi_{P(a)}$ is exactly $1$. Furthermore, we provide a combinatorial algorithm, based on Zelevinsky's elementary operations and the decomposition numbers $m(b;a)$, to compute the multiplicities of each $\pi_\mu$ in $St(\langle a\rangle)^{K_+}$.
\end{abstract}

\section{Introduction}
Branching laws, which describe the decomposition of a representation of a large group when restricted to a smaller subgroup, are a central theme in representation theory. Classical examples include the restriction from $GL_n$ to $GL_{n-1}$,  Bessel models, Jacobi models, and various other settings. However, the case where the smaller subgroup is a \emph{hyperspecial} subgroup $K$ (or its pro-unipotent radical $K_+$) has received far less attention, even for the depth-zero part, i.e., the space of $K_+$-fixed vectors $V^{K_+}$. For a representation $V$ of $GL_n(F)$ (where $F$ is a non-archimedean local field of characteristic $0$), the $K_+$-fixed subspace $V^{K_+}$ naturally carries an action of the finite Lie group $K/K_+ \cong GL_n(\mathbb{F}_q)$, where $\mathbb{F}_q$ is the residue field. Understanding the decomposition of $V^{K_+}$ into irreducible $K/K_+$-modules is therefore a natural branching problem, yet systematic results are sparse.

In this paper we study this problem for a large family of representations of $GL_n(F)$, namely those that are \emph{Iwahori-spherical} (i.e., possess a non-zero vector fixed by the standard Iwahori subgroup $I$). Such representations are known to be parameterized by combinatorial data via Zelevinsky's theory of segments \cite{Zelevinsky1980}. More precisely, let $\mathcal{C}$ be the set of irreducible cuspidal representations of the groups $GL_m(F)$ ($m\ge 1$), and let $\nu(g)=|\det g|$. A \emph{segment} $\Delta = [\rho,\nu^k\rho]$ is a set $\{\rho,\nu\rho,\dots,\nu^k\rho\}$. Zelevinsky associates to every segment $\Delta$ an irreducible representation $\langle\Delta\rangle$ (the unique irreducible submodule of the product $\rho\times\nu\rho\times\cdots\times\nu^k\rho$) and also a distinguished irreducible quotient $St(\langle\Delta\rangle)$, called the Zelevinsky dual. For a multiset of segments $a = \{\Delta_1,\dots,\Delta_r\}$, one forms the product $\langle\Delta_1\rangle\times\cdots\times\langle\Delta_r\rangle$; it has a unique irreducible submodule denoted $\langle a\rangle$. Moreover, every irreducible representation of $GL_n(F)$ is isomorphic to $\langle a\rangle$ for some multiset $a$ of segments whose total length sums to $n$. The Zelevinsky involution $St$ sends $\langle\Delta\rangle$ to $St(\langle\Delta\rangle)$ and extends to an involutive automorphism of the Grothendieck ring; in particular, every irreducible representation can also be written uniquely as $St(\langle a\rangle)$ for some $a$.

Our main object of study is the space of $K_+$-fixed vectors of the representation $St(\langle a\rangle)$, where the segments are assumed to consist of \emph{unramified characters} (so that the representation is Iwahori-spherical). For such representations, we prove a decomposition theorem for $St(\langle a\rangle)^{K_+}$ as a module over $K/K_+\cong GL_n(\mathbb{F}_q)$. The decomposition is governed by the dominance order on partitions and the Kostka numbers.

Let us first recall some known facts about the $K_+$-fixed subspace of Iwahori-spherical representations. In the author's previous work \cite{Wang2026}, it was shown that for any Iwahori-spherical representation $V$ of $GL_n(F)$, every irreducible $K/K_+$-subrepresentation of $V^{K_+}$ lies in the principal Harish-Chandra series of $K/K_+$, i.e., it appears as a composition factor of $\operatorname{Ind}_{I/K_+}^{K/K_+}(\mathbf{1})$. It is a classical fact that the irreducible representations in this principal series are parameterized by the irreducible representations of the Weyl group of $GL_n(F)$, which is the symmetric group $S_n$ (see e.g. \cite{GeckPfeiffer2000}). The irreducible representations of $S_n$ are in one-to-one correspondence with partitions $\lambda$ of $n$. We denote by $\pi_\lambda$ the irreducible $K/K_+$-module attached to the partition $\lambda$ under this correspondence.

Now consider the depth-zero part of an irreducible Iwahori-spherical representation $V$. A natural question, raised by Prasad in \cite{Prasad2025} (Question 2), asks for a classification of all irreducible Iwahori-spherical representations $V$ of $GL_n(F)$ such that a given $\pi_\lambda$ occurs in $V^{K_+}$, and he expects the answer to be expressed in terms of Zelevinsky's segment classification. More precisely, for an irreducible Iwahori-spherical representation $St(\langle a\rangle)$ (or equivalently $\langle a\rangle$), when does $\pi_\lambda$ appear in $St(\langle a\rangle)^{K_+}$? And how can the multiplicity be described using the combinatorial data of the multiset $a$?

In this paper we give a complete answer to Prasad's question for $GL_n$. For a multiset $a$ of segments (with unramified cuspidal supports), we define a partition $P(a)$ as the conjugate of the partition formed by the lengths of the segments. We prove:

\begin{theorem}
Let $V = St(\langle\Delta_1\rangle)\times \cdots \times St(\langle\Delta_r\rangle)$ be a generic Iwahori-spherical representation with segment lengths $n_1\ge\cdots\ge n_r$ and set $\lambda = (n_1,\dots,n_r)'$. Then
$$
V^{K_+} \cong \bigoplus_{\mu \trianglelefteq \lambda} K_{\mu'\lambda'}\,\pi_\mu,
$$
where $K_{\mu'\lambda'}$ is the Kostka number and $\mu\trianglelefteq\lambda$ denotes the dominance order. In particular, $\pi_\mu$ occurs in $V^{K_+}$ if and only if $\mu\trianglelefteq\lambda$.
\end{theorem}

For an arbitrary irreducible Iwahori-spherical representation $St(\langle a\rangle)$ (or equivalently $\langle a\rangle$ after applying Zelevinsky duality), we associate the partition $P(a)$. We show:

\begin{theorem}
If $\pi_\mu$ occurs in $St(\langle a\rangle)^{K_+}$, then necessarily $\mu \trianglelefteq P(a)$. And the multiplicity of $\pi_{P(a)}$ is exactly $1$. Moreover, we provide a combinatorial algorithm, based on Zelevinsky's elementary operations and the decomposition numbers $m(b;a)$, to compute the multiplicities of each $\pi_\mu$ in $St(\langle a\rangle)^{K_+}$.
\end{theorem}

Thus we obtain both a necessary condition and an explicit (though inductive) method to determine the full branching law for these representations.

The paper is organized as follows. In Section 2 we recall the necessary background on partitions, Kostka numbers, distinguished coset representatives, Hecke algebras, and the action of the Hecke algebra on $B$-fixed vectors. Section 3 reviews Zelevinsky's theory of segments and the classification of irreducible representations of $GL_n(F)$. In Section 4 we prove a key technical lemma that identifies the $B$-fixed vectors of a parabolically induced representation with a tensor product over the Hecke algebra. Section 5 applies this lemma to the specific case of Steinberg duals of segments and establishes the decomposition for generic Iwahori-spherical representations. Finally, Section 6 extends the results to arbitrary irreducible Iwahori-spherical representations and describes the combinatorial algorithm. To make the exposition self-contained, we have included an appendix in which we prove some relevant facts concerning the principal series and Iwahori-spherical representations.

\section{Preliminary}

\subsection{Partition}
In this subsection, we just recall some basic definitions and lemmas.
\begin{definition}[Partition]
A \textbf{partition} $\lambda = (\lambda_1, \lambda_2, \lambda_3, \dots)$ of $n$ is a sequence of non‑negative integers such that
$$
\lambda_1 \ge \lambda_2 \ge \lambda_3 \ge \cdots \quad \text{and} \quad \sum_{i=1}^{\infty} \lambda_i = n.
$$
\end{definition}

\begin{definition}[Dominance order]
If $\lambda$ and $\mu$ are partitions of $n$, we say that $\lambda$ \textbf{dominates} $\mu$, and write $\lambda \trianglerighteq \mu$, provided that
$$
\lambda_1 + \cdots + \lambda_i \ge \mu_1 + \cdots + \mu_i \qquad \text{for all } i \ge 1.
$$
If $\lambda \trianglerighteq \mu$ and $\lambda \neq \mu$, we write $\lambda \triangleright \mu$.
\end{definition}

\begin{definition}[Young diagram and $\lambda$-tableau]
If $\lambda$ is a partition of $n$, its \textbf{diagram} $[\lambda]$ is the set of nodes
$$
[\lambda] = \{ (i,j) \mid i,j \in \mathbb{Z},\; 1 \le i,\; 1 \le j \le \lambda_i \}.
$$
A \textbf{$\lambda$-tableau} is one of the $n!$ arrays of integers obtained by replacing each node in $[\lambda]$ by one of the integers $1,2,\dots,n$ without repetition.
\end{definition}

\begin{definition}[Conjugate partition]
If $[\lambda]$ is a diagram, the \textbf{conjugate diagram} $[\lambda']$ is obtained by interchanging the rows and columns in $[\lambda]$. $\lambda'$ is the partition of $n$ conjugate to $\lambda$.
\end{definition}

\begin{lemma}[Basic combinatorial lemma]
Let $\lambda$ and $\mu$ be partitions of $n$, and suppose that $t_1$ is a $\lambda$-tableau and $t_2$ is a $\mu$-tableau. Suppose that for every $i$ the numbers from the $i$-th row of $t_2$ belong to different columns of $t_1$. Then $\lambda \trianglerighteq \mu$.
\end{lemma}

\begin{lemma}[Duality of dominance order]
For partitions $\lambda$ and $\mu$ of $n$, we have
$$
\lambda \trianglerighteq \mu \quad \text{if and only if} \quad \mu' \trianglerighteq \lambda'.
$$
\end{lemma}

\subsection{Littlewood-Richardson coefficient and Kostka number}
It is well-known that the representations of the symmetric group $S_n$ can be parametrized by partitions. For any partition $\lambda$ of $n$, denote the representation of $S_n$ corresponding to $\lambda$ by $\rho_\lambda$. In particular, when $\lambda = (1^n)$, $\rho_\lambda$ is the sign representation; we denote this representation by $\operatorname{sgn}$.

Suppose $n = n_1 + n_2 + \dots + n_r$, and let $H = S_{n_1} \times \dots \times S_{n_r} \leq S_n$. For each $i$, let $\lambda_i$ be a partition of $n_i$. The Littlewood–Richardson coefficient is a useful tool for studying the decomposition of $\operatorname{Ind}_H^{S_n}(\rho_{\lambda_1} \otimes \dots \otimes \rho_{\lambda_r})$. There is a rule, called the Littlewood–Richardson rule, which describes this decomposition. However, the situation becomes simpler when each $\rho_{\lambda_i}$ is the sign representation. In the book by Geck and Pfeiffer \cite{GeckPfeiffer2000}, Section 6.6.3, the following Pieri rule is given:

\begin{theorem}
Suppose $n = k + l$ and $\lambda$ is a partition of $k$. Let $\operatorname{sgn}$ be the sign representation of $S_l$. Then we have the decomposition
\[
\operatorname{Ind}_{S_k \times S_l}^{S_n}(\rho_\lambda \otimes \operatorname{sgn}) = \sum_{\nu} \rho_\nu,
\]
where the sum runs over all partitions $\nu$ of $n$ whose Young diagram is obtained from $\lambda$ by adding $l$ boxes, with no two boxes in the same row.
\end{theorem}

Now consider $n = n_1 + n_2 + \dots + n_r$ with $n_1 \ge n_2 \ge \dots \ge n_r$. Let $\lambda = (n_1, n_2, \dots, n_r)'$ (the conjugate partition) and $H = S_{n_1} \times S_{n_2} \times \dots \times S_{n_r} \le S_n$. For each $i$, let $\operatorname{sgn}_i$ be the sign representation of $S_{n_i}$. Using the Pieri rule repeatedly, we can compute the decomposition of $\operatorname{Ind}_H^{S_n}(\operatorname{sgn}_1 \otimes \dots \otimes \operatorname{sgn}_r)$.

First, we start with the diagram $(1^{n_1})$. Then we add $n_2$ boxes to the previous diagram, respecting the condition that no two boxes lie in the same row. We repeat this process $r-1$ times. Finally we obtain the decomposition
\[
\operatorname{Ind}_H^{S_n}(\operatorname{sgn}_1 \otimes \dots \otimes \operatorname{sgn}_r) = \sum_{\tau} c_\tau \rho_\tau,
\]
where $\tau$ runs over all partitions of $n$. Actually, it is well-known that $c_\tau$ is just the Kostka number $K_{\tau'\lambda'}$.

\begin{lemma}
$c_\tau \neq 0$ if and only if $\tau$ is dominated by $\lambda$, and $c_\lambda = 1$.
\end{lemma}
\begin{proof}
It is just the properties of Kostka number, but we should note that these properties are not trivial.
\end{proof}

\subsection{Distinguished left coset representatives}

Let $(W,S)$ be a Coxeter system, where $W$ is a finite Coxeter group and $S$ is its set of generators. For a subset $J\subseteq S$, we denote by $W_J = \langle J \rangle$ the corresponding parabolic subgroup; then $(W_J, J)$ is again a Coxeter system, and the length function on $W_J$ coincides with the restriction of the length function $\ell$ on $W$.

The following proposition provides a canonical set of representatives for the left cosets of $W_J$ in $W$.

\begin{theorem}
For $J\subseteq S$, set
$$
Y_J := \{ x \in W \mid \ell(xs) > \ell(x) \text{ for all } s \in J \}.
$$
\begin{enumerate}
\item[(a)] Every element $w \in W$ can be written uniquely as $w = x v$ with $x \in Y_J$ and $v \in W_J$; moreover $\ell(w) = \ell(x) + \ell(v)$.
\item[(b)] For $x \in W$, the following are equivalent:
  \begin{itemize}
  \item[(i)] $x \in Y_J$;
  \item[(ii)] $\ell(x v) = \ell(x) + \ell(v)$ for all $v \in W_J$;
  \item[(iii)] $x$ is the unique element of minimal length in the left coset $x W_J$.
  \end{itemize}
  In particular, $Y_J$ is a complete set of representatives for the left cosets of $W_J$ in $W$.
\end{enumerate}
\end{theorem}

\begin{definition}[Distinguished left coset representatives]
The set $Y_J$ defined in the above theorem is called the set of \emph{distinguished left coset representatives} of $W_J$ in $W$. We write
$$
W = Y_J \cdot W_J
$$
to indicate that the multiplication map $Y_J \times W_J \to W$, $(x,v) \mapsto x v$, is a bijection and that $\ell(x v) = \ell(x) + \ell(v)$ for all $x \in Y_J$, $v \in W_J$.
\end{definition}

The following lemma is a crucial tool for understanding how distinguished representatives behave under left multiplication by a generator.

\begin{lemma}[Deodhar's lemma]
Let $J\subseteq S$, let $x \in Y_J$ and let $s \in S$. Then either $sx \in Y_J$ or $sx = x u$ for some $u \in J$.
\end{lemma}

\subsection{Parabolic induction for Hecke algebras}

Let $(W,S)$ be a finite Coxeter system and let $\mathbf{H}_q$ be the Iwahori–Hecke algebra of $W$ with parameter $q$ (a prime power), specialised over $\mathbb{C}$.  
We assume that $\mathbf{H}_q$ is semisimple (this holds for generic $q$, e.g. when $q$ is not a root of unity; in particular it is true for the Hecke algebras arising from $p$-adic groups).  
All modules are left modules.  
For a subset $J\subseteq S$, denote by $(\mathbf{H}_q)_J$ the parabolic subalgebra generated by $\{T_s\mid s\in J\}$.

\begin{definition}[Induced module, cf. Definition 9.1.1 \cite{GeckPfeiffer2000}]
Let $V$ be a left $(\mathbf{H}_q)_J$-module. The induced module is defined as
\[
\operatorname{Ind}_J^S(V) := \mathbf{H}_q \otimes_{(\mathbf{H}_q)_J} V,
\]
where $\mathbf{H}_q$ is viewed as a $(\mathbf{H}_q, (\mathbf{H}_q)_J)$-bimodule (left multiplication by $\mathbf{H}_q$, right multiplication by $(\mathbf{H}_q)_J$).  
The action of $\mathbf{H}_q$ on $\operatorname{Ind}_J^S(V)$ is given by $h\cdot (h'\otimes v) = (hh')\otimes v$.
\end{definition}

To describe the structure of $\operatorname{Ind}_J^S(V)$ we use the set $Y_J$ of distinguished left coset representatives of $W_J$ in $W$ (see Section 2).  
The algebra $\mathbf{H}_q$ is free as a right $(\mathbf{H}_q)_J$-module with basis $\{T_x\mid x\in Y_J\}$.  
Consequently, every element of $\operatorname{Ind}_J^S(V)$ can be written uniquely as $\sum_{x\in Y_J} T_x \otimes v_x$ with $v_x\in V$.

The action of the generators $T_s$ ($s\in S$) on $\operatorname{Ind}_J^S(V)$ is given by explicit rules.  
For $x\in Y_J$ and $v\in V$, we have the following formulas (cf. the first displayed formula on page 287 of \cite{GeckPfeiffer2000}, adapted to left modules):

\begin{lemma}[Action of generators on induced modules]
With the notation above,
\[
T_s \cdot (T_x \otimes v) = 
\begin{cases}
T_{sx} \otimes v, & \text{if } sx\in Y_J \text{ and } \ell(sx)>\ell(x);\\[4pt]
T_x \otimes (T_u \cdot v), & \text{if } sx = x u \text{ for some } u\in J;\\[4pt]
q\, T_{sx}\otimes v + (q-1)\, T_x\otimes v, & \text{if } sx\in Y_J \text{ and } \ell(sx)<\ell(x).
\end{cases}
\]
\end{lemma}

The compatibility of induction with the specialisation $q=1$ (which sends $\mathbf{H}_q$ to the group algebra $\mathbb{C}[W]$) is established in Section 9.1.9 of \cite{GeckPfeiffer2000}.  
This yields the following fundamental result.

\begin{theorem}[Decomposition numbers for parabolic induction]
\label{thm:induction_decomposition}
For any parabolic subgroup $W_J\subseteq W$ and any irreducible representation $\psi$ of $(\mathbf{H}_q)_J$, the multiplicities of irreducible $\mathbf{H}_q$-modules in $\operatorname{Ind}_J^S(\psi)$ are exactly the same as the multiplicities of irreducible $W$-modules in $\operatorname{Ind}_{W_J}^W(\psi|_{W_J})$.
\end{theorem}

Thus, in the semisimple case (which holds for the Hecke algebras arising from $p$-adic groups), the parabolic induction rules for $\mathbf{H}_q$ are identical to those for the finite Coxeter group $W$.

\subsection{The action of the Hecke algebra}

Even though the discussion of this subsection can be adapted to a larger setting, for simplicity we just consider $G = GL_n(\mathbb{F}_q)$. Let $U_\alpha$ be the root group corresponding to a root $\alpha$, and let $B$ be the standard Borel subgroup of $G$.

For every representation $V$ of $G$ all of whose irreducible components lie in the Harish-Chandra principal series, we can take the $B$-fixed vectors $V^B$ and then make it into a Hecke algebra module. The finite Hecke algebra associated to $G$ is $\operatorname{End}_G(\operatorname{Ind}_B^G(\mathbf{1}))^{\mathrm{opp}}$. 

We should  mention a well-known result \cite{GeckJacon2011}:
\begin{theorem}
The functor $\Theta: M \longmapsto \operatorname{Hom}_{G}(\operatorname{Ind}_{B}^{G}\mathbf{1},M)$ induces a bijection from irreducible subrepresentations of the principal Harish-Chandra series to the simple Hecke algebra modules up to isomorphism. The action of this $H$-module is $h*f=f \circ h$ for $h\in H$ and $f \in \operatorname{Hom}_{G}(\operatorname{Ind}_{B}^{G}\mathbf{1},M)$.
\end{theorem}

We can go one step further using Frobenius reciprocity:
$$
\operatorname{Hom}_{G}(\operatorname{Ind}_{B}^{G}\mathbf{1},M)=\operatorname{Hom}_{B}(\mathbf{1},M)=M^{B}.
$$
Hence if we have a finite dimensional representation $V$ of $G$, and we assume further that every irreducible component of $V$ is a subrepresentation of $\operatorname{Ind}_{B}^{G}\mathbf{1}$, then we can take $B$-invariant vectors of $V$ to make it into an $H$-module. Also note that ${Ind}_{B}^{G}\mathbf{1}=\mathbb{C}[G/B]$.

Then for an arbitrary element $n$ in the normalizer $N_{G}(T)$, we can define an element of $H$:
$$
T_n:\operatorname{Ind}_{B}^{G}\mathbf{1} \longrightarrow \operatorname{Ind}_{B}^{G}\mathbf{1}
$$
which sends $xB$ to the sum of all cosets $yB$ such that $x^{-1}y \in BnB$. We also define $q_w=[BwB:B]$ for an arbitrary element $w$ in the Weyl group. Then we have:

\begin{theorem}
$\{T_w\mid w\in W\}$ forms a basis of the Iwahori–Hecke algebra $H$. If we further assume that $S$ is the set of simple reflections, then $\{T_w\mid w \in W\}$ forms a basis of $H$, and $H$ has the following presentation:
$$
H = \left\langle T_s\;(s\in S) \;\middle|\; 
\underbrace{T_sT_tT_s\cdots}_{m_{st}\text{ factors}} = \underbrace{T_tT_sT_t\cdots}_{m_{st}\text{ factors}},\;
(T_s - q_s)(T_s +1) = 0 \right\rangle.
$$
\end{theorem}

The action of this Hecke algebra has already been written down in an explicit way. We just state the explicit formula here. And we will also give a  proof in the appendix.

\begin{lemma}
Suppose that $V$ is a representation of $G$ every irreducible component of which lies in the principal Harish-Chandra series of $G$. Suppose also that $\alpha$ is a simple root of $G$. Then $V^B$ can be regarded as a Hecke module, and for every $v \in V^B$ we have
$$
T_{s_\alpha} v = \sum_{u \in B s_\alpha B / B} u v.
$$
\end{lemma}

Note that $B s_\alpha B = U_\alpha s_\alpha B$, and $\{ u s_\alpha \mid u \in U_\alpha \}$ is a set of coset representatives of $B s_\alpha B / B$. Hence the action of $T_{s_\alpha}$ can also be written as
$$
T_{s_\alpha} v = \sum_{u \in U_\alpha} u s_\alpha v.
$$

The above theorem tells us that we can use the representation of the Weyl group to represent the subrepresentation which lies in the Harish-Chandra principal series. And in the $GL_n$ case, we can use a partition of $n$ to represent the representations which lie in the principal Harish-Chandra series. Hence for a partition $\lambda$, we use $\pi_\lambda$ to denote the corresponding representation.

\section{Review of Zelevinsky's segments}

We recall some fundamental notions from \cite{Zelevinsky1980}. Let $\mathbf{F}$ be a non-archimedean local field and $\mathbf{G}_n = \mathrm{GL}(n,\mathbf{F})$. For representations $\pi_i \in \mathrm{Alg}\,\mathbf{G}_{n_i}$ ($i=1,\dots,r$) we denote by
\[
\pi_1 \times \cdots \times \pi_r = i_{(n_1,\dots,n_r)}(\pi_1 \otimes \cdots \otimes \pi_r)
\]
the induced representation from the parabolic subgroup of block-upper-triangular matrices; this is called the \emph{product} of the $\pi_i$.

\subsection{Segments and associated representations}
Let $\mathcal{C}$ be the set of equivalence classes of \emph{irreducible cuspidal representations} of the groups $\mathbf{G}_n$ ($n\ge 1$). A \emph{segment} in $\mathcal{C}$ is a subset of the form
\[
\Delta = \{\rho,\nu\rho,\nu^2\rho,\dots,\nu^k\rho = \rho'\},
\]
where $\rho\in\mathcal{C}$ is cuspidal, $\nu(g)=|\det g|$ and $k\ge 0$; we write $\Delta = [\rho,\rho']$. For such a segment, Zelevinsky constructs an irreducible representation $\langle\Delta\rangle\in \mathrm{Irr}\,\mathbf{G}_{km}$ (where $\rho\in\mathrm{Irr}\,\mathbf{G}_m$) as the unique irreducible submodule of $\rho\times\nu\rho\times\cdots\times\rho'$. Moreover, there is also a distinguished irreducible quotient, denoted $St(\langle\Delta\rangle)$, which is the unique irreducible quotient of the same product. By Proposition~3.4, the restriction (Jacquet module) of $\langle\Delta\rangle$ to a Levi subgroup satisfies
\[
r_{(l,n-l),(n)}(\langle\Delta\rangle) = 
\begin{cases}
0 & \text{if } l \text{ is not a multiple of } m,\\
\langle[\rho,\nu^{p-1}\rho]\rangle \otimes \langle[\nu^p\rho,\rho']\rangle & \text{if } l = mp,\; 0\le p\le k.
\end{cases}
\]

\subsection{Linked segments and irreducibility criterion}
Two segments $\Delta_1 = [\rho_1,\rho_1']$ and $\Delta_2 = [\rho_2,\rho_2']$ are called \emph{linked} if $\Delta_1\not\subset \Delta_2$, $\Delta_2\not \subset \Delta_1$ and $\Delta_1\cup\Delta_2$ is again a segment. The following criterion is Theorem~4.2 of \cite{Zelevinsky1980}:

\begin{theorem}[Zelevinsky, Theorem 4.2]
Let $\Delta_1,\dots,\Delta_r$ be segments in $\mathcal{C}$. The product $\langle\Delta_1\rangle\times\cdots\times\langle\Delta_r\rangle$ is irreducible if and only if no two of the segments $\Delta_i,\Delta_j$ are linked.
\end{theorem}

\subsection{Classification of irreducible representations}
Let $\mathcal{O}$ be the set of all finite multisets of segments. For $a = \{\Delta_1,\dots,\Delta_r\}\in\mathcal{O}$, choose an ordering such that $\Delta_i$ does not precede $\Delta_j$ for $i<j$ (i.e. no $\Delta_i$ is linked with a later one in a way that would create a “preceding” relation). Then the product $\langle\Delta_1\rangle\times\cdots\times\langle\Delta_r\rangle$ has a unique irreducible submodule, denoted $\langle a\rangle$. Theorem~6.1 states:

\begin{theorem}[Zelevinsky, Theorem 6.1]
\begin{enumerate}
\item Let $\Delta_1,\ldots,\Delta_r$ be segments in $\mathcal{C}$. Suppose for each pair of indices $i,j$ such that $i<j$, $\Delta_i$ does not precede $\Delta_j$. Then the representation $\langle\Delta_1\rangle\times\ldots\times\langle\Delta_r\rangle$ has a unique irreducible submodule; denote it by $\langle\Delta_1,\ldots,\Delta_r\rangle$.
\item The representations $\langle\Delta_1,\ldots,\Delta_r\rangle$ and $\langle\Delta_1',\ldots,\Delta_s'\rangle$ are isomorphic if and only if the sequences $(\Delta_1,\ldots,\Delta_r)$ and $(\Delta_1',\ldots,\Delta_s')$ are equal up to a rearrangement.
\item Any irreducible representation of $\mathbf{G}_n$ is isomorphic to some representation of the form $\langle\Delta_1,\ldots,\Delta_r\rangle$.
\end{enumerate}
\end{theorem}

\subsection{Decomposition numbers and partial order}
For $a,b\in\mathcal{O}$, let $m(b;a)$ be the multiplicity of $\langle b\rangle$ in the Jordan–Hölder series of $\pi(a)=\langle\Delta_1\rangle\times\cdots\times\langle\Delta_r\rangle$. An \emph{elementary operation} on a multiset $a$ consists of replacing two linked segments $\Delta,\Delta'$ by $\Delta\cup\Delta'$ and $\Delta\cap\Delta'$ (the latter is dropped if empty). Write $b\le a$ if $b$ can be obtained from $a$ by a sequence of elementary operations. Theorem~7.1 gives:

\begin{theorem}[Zelevinsky, Theorem 7.1]
$m(b;a)\neq 0$ if and only if $b\le a$. Moreover $m(a;a)=1$ for every $a\in\mathcal{O}$.
\end{theorem}

\subsection{Duality and non-degenerate representations}
There is an involutive automorphism $\omega\mapsto St(\omega)$ of the Grothendieck ring $\mathcal{R}$ (the representation bialgebra) uniquely determined by sending $\langle\Delta\rangle$ to $St(\langle\Delta\rangle)$ for every segment $\Delta$. This duality exchanges submodules and quotients. For non-degenerate representations, Theorem~9.7 gives a classification in terms of the representations $St(\langle\Delta\rangle)$:

\begin{theorem}[Zelevinsky, Theorem 9.7]
\begin{enumerate}
\item For any $a=\{\Delta_1,\dots,\Delta_r\}\in\mathcal{O}$, the product $St(\langle\Delta_1\rangle)\times\cdots\times St(\langle\Delta_r\rangle)$ is non-degenerate. It is irreducible if and only if no two of the segments $\Delta_1,\dots,\Delta_r$ are linked.
\item Every irreducible non-degenerate representation $\omega$ of $\mathbf{G}_n$ can be written uniquely as a product $\omega = St(\langle\Delta_1\rangle)\times\cdots\times St(\langle\Delta_r\rangle)$ where the segments $\Delta_i$ are pairwise not linked.
\end{enumerate}
\end{theorem}

\section{Key Lemma}
Suppose that $F$ is a non-Archimedean local field of characteristic $0$, and let $G = GL_n(F)$. Let $K = GL_n(\mathfrak{o})$ be its hyperspecial subgroup, and $K_+$ its pro-unipotent radical. Let $I$ be its standard Iwahori subgroup. Mishra and Rösner established an isomorphism \cite{MishraRoesner2017}.

\begin{theorem}
Suppose that $P = MN$ is a standard parabolic subgroup of $G$, and $(\pi, V)$ is a representation of $M$. Then there is a natural isomorphism (as representations of $K/K_+$):
$$
\operatorname{Ind}_P^G(V)^{K_+}  \cong \operatorname{Ind}_{(P \cap K)/(P \cap K_+)}^{K/K_+} (V^{M \cap K_+}).
$$
\end{theorem}

We should make a comment here: since we will restrict the representation to a compact subgroup, normalization does not matter for this result.

Here comes the key lemma of this article. Even though this lemma can be adapted to a larger setting, for simplicity we just consider $GL_n(\mathbb{F}_q)$. The isomorphism at the vector space level is quite clear, but the isomorphism between Hecke algebra modules is quite technical.

\begin{lemma}
Suppose that $n=n_1+n_2+\cdots+n_r$, and $\rho_i$ is a representation of $GL_{n_i}(\mathbb{F}_q)$ which lies in the principal Harish-Chandra series. Suppose that $B$ is the standard Borel subgroup of $GL_n(\mathbb{F}_q)$, and $B_i$ is the standard Borel subgroup of $GL_{n_i}(\mathbb{F}_q)$. Let $M=GL_{n_1}\times GL_{n_2}\times\cdots\times GL_{n_r}$ be the standard Levi of $GL_n$. Also assume the corresponding standard parabolic subgroup is $P$. Suppose that the finite Hecke algebra of $GL_n$ is $H_n$, and the corresponding parabolic sub-Hecke-algebra of $M$ is $H_M$. Then we have an $H_n$-module isomorphism:
$$
\operatorname{Ind}_{P}^{GL_n}(\rho_1\otimes\rho_2\otimes\cdots\otimes\rho_r)^B \cong H_n \otimes_{H_M}\bigl(\rho_1^{B_1}\otimes \rho_2^{B_2}\otimes\cdots\otimes\rho_r^{B_r}\bigr),
$$
where the left-hand side has an $H_n$-action via right translation, and the right-hand side has a natural left $H_n$-action.
\end{lemma}

\begin{proof}
Let us denote the Weyl group of $GL_n$ by $W$, and the Weyl group of $M$ by $W_M$. 
Let $X$ be the set of distinguished coset representatives of the left coset $W/W_M$. Denote by $J$ the subset of simple roots corresponding to $M$.

Define a map $\phi$ by $\phi(f)=\sum_{w\in X} T_w\otimes f(w^{-1})$ for $f$ in the left-hand side.

\paragraph{Step 1: $\phi$ is well-defined.}
We only need to show that $f(w^{-1})\in \rho_1^{B_1}\otimes \rho_2^{B_2}\otimes\cdots\otimes\rho_r^{B_r}$. For an arbitrary element $b\in B_1\times B_2\times\cdots\times B_r$, we have $b f(w^{-1})=f(b w^{-1})=f(w^{-1} w b w^{-1})$. Note that by the definition of $X$, we have $\ell(ws)=\ell(w)+1$ for every $w\in X$ and $s\in J$. Hence $w(J)>0$, and therefore $w(B_1\times B_2\times\cdots\times B_r)w^{-1}\subset B$. Consequently $b f(w^{-1})=f(w^{-1} w b w^{-1})=f(w^{-1})$. So $\phi$ is well-defined. Using the above argument, we can also see a useful observation: if $w \in X$, then $f$ is constant on $Bw^{-1}B$. Suppose that $P=MN$, and $b \in B \subset P$ has Levi decomposition $b=vu$, where $u \in M \cap B$, and $v \in N$. Then for any $b' \in B$, we have $f(bw^{-1}b')=f(vuw^{-1}b')=f(w^{-1}wuw^{-1}b')=f(w^{-1})$.

\paragraph{Step 2: $\phi$ is an isomorphism of vector spaces.}
This is a standard argument using Mackey's theorem. To see this, one should be aware of a standard fact about Hecke algebras: for every $H_M$-module $V$, as vector spaces,
$$
H_n \otimes_{H_M} V \cong \bigoplus_{x\in X} T_x \otimes_{H_M} V.
$$

\paragraph{Step 3: $\phi$ is an $H_n$-module homomorphism.}
For every $w\in X$ and a simple root $\alpha$, according to Deodhar's lemma there are three cases:
\begin{enumerate}
\item $s_\alpha w \notin X$ and there exists $\beta\in J$ such that $s_\alpha w = w s_\beta$;
\item $s_\alpha w \in X$ and $\ell(s_\alpha w)=\ell(w)+1$;
\item $s_\alpha w \in X$ and $\ell(s_\alpha w)=\ell(w)-1$.
\end{enumerate}
We have
$$
\phi(T_{s_\alpha}f)=\sum_{x\in X}\sum_{u\in U_{\alpha}} T_x \otimes f(x^{-1} u s_\alpha).
$$

\paragraph{Case 1.} Since $w\in X$ and $\beta\in J$, we have $\ell(w s_\beta)=\ell(w)+1$, which means $w(\beta)>0$, and because $w s_\beta w^{-1}=s_{w(\beta)}=s_\alpha$, we deduce $w(\beta)=\alpha$. The coefficient on the left-hand side corresponding to $T_w$ is
$$
\sum_{u\in U_\alpha} f(w^{-1} u s_\alpha)
$$
$$
= \sum_{u\in U_\alpha} f(w^{-1} u w w^{-1} s_\alpha)
$$
$$
= \sum_{v\in U_{w^{-1}(\alpha)}} f(v w^{-1} s_\alpha)
$$
$$
= \sum_{v\in U_\beta} f(v w^{-1} s_\alpha)
$$
$$
= \sum_{v\in U_\beta} f(v s_\beta w^{-1})
$$
$$
= \sum_{v\in U_\beta} v s_\beta f(w^{-1})
$$
$$
= T_{s_\beta} f(w^{-1}).
$$

Notice that
$$
T_{s_\alpha} T_w \otimes f(w^{-1}) = T_w \otimes T_{s_\beta} f(w^{-1}),
$$
so the coefficients of $T_w$ coincide on both sides.

\paragraph{Case 2.} Consider the coefficients of $T_w$ and $T_{s_\alpha w}$ on the left-hand side:
$$
\sum_{u\in U_\alpha} \bigl( T_w \otimes f(w^{-1} u s_\alpha) + T_{s_\alpha w} \otimes f(w^{-1} s_\alpha u s_\alpha) \bigr).
$$
By the property of a BN‑pair, $w^{-1} U_\alpha s_\alpha \subset B w^{-1} s_\alpha B$ because $\ell(w^{-1}s_\alpha)=\ell(w^{-1})+1$. Hence
$$
\sum_{u\in U_\alpha} f(w^{-1} u s_\alpha) = q f(w^{-1} s_\alpha).
$$
Again by the BN‑pair property,
$$
B w^{-1} s_\alpha U_\alpha s_\alpha B = B w^{-1} s_\alpha B \cup B w^{-1} B,
$$
since $\ell(w^{-1}s_\alpha)=\ell(w)+1$. Combining the observation below Step 1, we have:
$$
\sum_{u\in U_\alpha} f(w^{-1} s_\alpha u s_\alpha) = f(w^{-1}) + (q-1) f(w^{-1} s_\alpha).
$$
Now consider the right-hand side:
$$
T_{s_\alpha}\bigl( T_w \otimes f(w^{-1}) + T_{s_\alpha w} \otimes f(w^{-1}s_\alpha) \bigr)
$$
$$
= T_{s_\alpha w}\otimes\bigl(f(w^{-1})+(q-1)f(w^{-1}s_\alpha)\bigr) + T_w \otimes\bigl(q f(w^{-1}s_\alpha)\bigr).
$$
Hence the coefficients coincide.

\paragraph{Case 3.} Replace $w$ by $w s_\alpha$ and repeat the argument of Case 2.
\end{proof}

We should mention here that $I/K_+$ is a Borel subgroup of $K/K_+$. Now, let us assume that $\chi$ is an unramified character of $G$, and consider a single segment $\langle \chi,\chi v,\dots,\chi v^k\rangle = \langle\Delta\rangle$. Section 2 of Zelevinsky's paper told us that the Jacquet functor $r_{(1,1,\dots,1)(n)}(St(\langle\Delta\rangle))\cong \chi v^k \otimes \dots \otimes \chi$. It is well-known that there is a vector space isomorphism (see Casselman's paper \cite{Casselman1980}) $V^I \cong V_N^{T_0}$. Hence $St(\langle\Delta\rangle)^I \cong (r_{(1,1,\dots,1)(n)}(St(\langle\Delta\rangle)))^{T_0}\cong(\chi v^k \otimes \dots \otimes \chi)^{T_0}$.
Hence $\dim_{\mathbb{C}} St(\langle\Delta\rangle)^I = 1$. we should note that:

\begin{theorem}
Suppose $G$ is an unramified reductive $p$-adic group, and $(\pi,V)$ is an Iwahori-spherical irreducible representation of $G$. Fix an Iwahori subgroup $I$ which is contained in $K$. Then every irreducible $K/K_{+}$-subrepresentation $W$ of $V^{K_{+}}$ lies in the principal Harish-Chandra series $\operatorname{Ind}_{I/K_{+}}^{K/K_{+}}(1)$ (i.e., parabolic induction from the trivial representation of a quasi-split torus).
\end{theorem}
We will give a self-contained proof of this proposition in the appendix.
And Corollary 4.5 of Chan-Savin \cite{ChanSavin2018} also indicated that $St(\langle\Delta\rangle)^{K_+}$ contains the Steinberg representation, which means $St(\langle\Delta\rangle)^I$ contains the sign representation. Hence, considering the dimension, we can see that $St(\langle\Delta\rangle)^{K_+} \cong St$. Here $St$ means the Steinberg representation of the finite Lie group $K/K_+$.

\begin{lemma}
$St(\langle\Delta\rangle)^{K_+} \cong St$ as $K/K_+$ representations. Also there is a Hecke algebra isomorphism $St(\langle\Delta\rangle)^I \cong \operatorname{sgn}$.
\end{lemma}

\section{Decomposition of the generic Iwahori-spherical representation}
By Theorem 3.4, every non-degenerate Iwahori-spherical representation of $G$ has the form $V = St(\langle\Delta_1\rangle) \times St(\langle\Delta_2\rangle) \times \cdots \times St(\langle\Delta_r\rangle)$ (not necessarily irreducible) where the cuspidal support of $V$ consists of characters. Also note that the order does not matter for the decomposition since they have the same semisimplification. Let us study the component of $V^{K_+}$. By Theorem 4.3 and Theorem 2.15, we only need to study $V^I$. Suppose that the corresponding parabolic is $P=MN$, and the corresponding parabolic Hecke subalgebra is $H_M$. Combining Theorem 4.1 and Lemma 4.2, we obtain a Hecke algebra isomorphism:

\begin{lemma}
$V^I = (St(\langle\Delta_1\rangle) \times St(\langle\Delta_2\rangle) \times \cdots \times St(\langle\Delta_r\rangle))^I \cong H_n \otimes_{H_M}(\operatorname{sgn}_1 \otimes \operatorname{sgn}_2 \otimes \cdots \otimes \operatorname{sgn}_r)$.
\end{lemma}
We should mention  that $H(K,I)$ is isomorphic to $H_n$ and that their actions on $V^I$ coincide. The readers can see this in the appendix. However, one can also first apply Theorem 4.1 to reduce the problem to the finite Lie group level, and then use only the theory of finite Lie groups together with Lemma 4.2. This approach avoids any potential ambiguity concerning the Hecke algebra actions.

Combining Lemma 2.8, Theorem 2.14 and Lemma 5.1 we can get the following theorem:

\begin{theorem}
Suppose that the length of $\Delta_i$ is $n_i$, and $n_1 \ge n_2 \ge \dots \ge n_r$. Let us denote the partition $(n_1,n_2,\dots,n_r)'$ (conjugate) by $\lambda$. Let $V = St(\langle\Delta_1\rangle) \times St(\langle\Delta_2\rangle) \times \cdots \times St(\langle\Delta_r\rangle)$ be an Iwahori-spherical representation of $G$ (not necessarily irreducible). Then $V^{K_+}$ can be decomposed as $\bigoplus_{\mu \trianglelefteq \lambda} K_{\mu'\lambda'}\,\pi_\mu$, where $K_{\mu'\lambda'}$ is the Kostka number. In particular, $\pi_\mu$ occurs in $V^{K_+}$ iff $\mu \trianglelefteq \lambda$.
\end{theorem}

\section{General case}
Suppose that $a$ is a multi-segment in Zelevinsky's sense, and suppose further that every element in $a$ is one-dimensional and unramified. Then $\langle a \rangle$ is an irreducible Iwahori-spherical representation of $G$. It is known from \cite{Zelevinsky1980} that the Grothendieck ring $\mathcal{R}$ of $G$ is generated by the set $\{ \langle \Delta \rangle \mid \Delta \text{ is a single segment} \}$ under addition and multiplication. Hence $\langle a \rangle$ can be expressed as a linear combination of products $\langle \Delta_1 \rangle \times \langle \Delta_2 \rangle \times \cdots \times \langle \Delta_r \rangle$ with integer coefficients. We can analyze this more precisely.

\begin{lemma}
If $a = \{\Delta_1,\Delta_2,\dots,\Delta_r\}$, and suppose that the length of $\Delta_i$ is $n_i$, and the cuspidal support of $a$ consists of unramified characters. Then there exist integers $c_b$ such that $\langle a \rangle = \sum_{b \le a} c_b \, \pi(b)$, and $c_a = 1$.
\end{lemma}

\begin{proof}
If $a$ is minimal, i.e., we cannot use an elementary operation to turn $a$ into a different element, then this means that segments in $a$ are pairwise not linked. Hence $\langle a \rangle = \pi(a)$. For a general $a$, by Theorem 3.3 we have $\langle a \rangle - \pi(a) = \sum_{b < a} -m(b;a)\langle b \rangle$. Then an inductive argument works.
\end{proof}

Suppose $\langle a \rangle = \sum_{b \le a} c_b \, \pi(b)$. Then we can apply the Zelevinsky involution: $St(\langle a \rangle) = \sum_{b \le a} c_b \, St(\pi(b))$. Since $St$ is an involution, every irreducible representation can be represented as $St(\langle a \rangle)$ (see \cite{Zelevinsky1980}). Indeed, Schneider and Stuhler \cite{SchneiderStuhler1997} proved that $St$ sends irreducible representations to irreducible representations.

For an arbitrary multi-segment $c = \{\Delta_1,\Delta_2,\dots,\Delta_r\}$, let $n_i$ be the length of $\Delta_i$, and assume $n_1 \ge n_2 \ge \dots \ge n_r$. We define the associated partition by $P(c) = (n_1,n_2,\dots,n_r)'$ (the conjugate partition). It is not hard to see that if $a \ge b$ then $P(a) \trianglerighteq P(b)$.

\begin{theorem}
$\pi_\lambda$ occurs in $(St(\langle a \rangle))^{K_+}$ only if $\lambda \trianglelefteq P(a)$. Moreover, the multiplicity of $\pi_{P(a)}$ is exactly $1$. Furthermore, the above argument also gives a **combinatorial algorithm** to compute the multiplicity of $\pi_\lambda$ in $(St(\langle a \rangle))^{K_+}$.
\end{theorem}

\begin{proof}
This is a direct corollary of Theorem 5.2 and Lemma 6.1.
The algorithm can be as follows: First, we determine the order (Section 3.4) of the multi-segments with the same cuspidal support and work out the decomposition number $m(a,b)$. This can be done because there are only finitely many multi-segments with the given cuspidal support. And there exists a combinatorial algorithm to compute $m(a,b)$. Then we compute the numbers $c_b$ in Lemma 6.1 via $m(a,b)$. Hence we obtain the decomposition.
\end{proof}
Remark: Applying dual again, we can also see that there is a minimal one with multiplicity 1. More precisely, if $<b>=St(<a>)$. Then the maxiaml one is $\pi_{P(a)}$, and the minimal one is $\pi_{P(b)'}$.
\\Note that this theorem gives only a necessary condition. To decide whether $\pi_\lambda$ actually occurs in $(St(\langle a \rangle))^{K_+}$, one must compute the multiplicity. The computation involves the numbers $m(b;a)$, which are very important in representation theory and are deeply related to the Kazhdan–Lusztig polynomials. At least theoretically, all these multiplicities can be computed, and we have provided a combinatorial algorithm to do so.
\appendix
\section{Harish-Chandra principal series and unramified principal series}
In this appendix, we prove some facts about Harish-Chandra principal series  and  unramified principal series. Through out the appendix, we assume $G$ is an unramified group, $K$ is its  hyperspecial subgroup, $K_+$ is the corresponding pro-unipotent  radical. $I$ is an Iwahori subgroup containing in $K$.
\begin{lemma}
\label{lem:cuspidal-support}
Let $G$ be an unramified reductive $p$-adic group, let $K$ be a hyperspecial subgroup of $G$, and let $K_{+}$ be its corresponding pro-unipotent radical. Let $(\pi,V)$ be an irreducible smooth depth-zero representation of $G$. Then we may regard $V^{K_{+}}$ naturally as a representation of $K$ and hence also as a representation of the finite reductive group $K/K_{+}$. Under these assumptions, then for an arbitrary irreducible $K/K_{+}$-subrepresentation $W$ of $V^{K_{+}}$, there exists a facet $c$ and a cuspidal representation $\rho$  of the finite Lie group $G_c/G_{c+}$, such that the vertex associated to $K$ is contained in the closure of $c$ and W contains $\rho$. (We can also prove that $[c,\rho]$ is a unrefined minimal K-type of V ). Moreover, every irreducible component of the $K/K_{+}$ representation is a subrepresentation of a Harish-Chandra series of the finite Lie group $K/K_{+}$.
\end{lemma}

\begin{proof}
Since $V$ has depth zero by hypothesis, the space $V^{K_{+}}$ is non-zero.

Let $W$ be an irreducible $K/K_{+}$-subrepresentation of $V^{K_{+}}$.

By Theorem 5.2 \cite{MoyPrasad1994} , $W$ contains a depth-zero minimal $K$-type associated to a parahoric subgroup $G_{c}\subseteq K$ (where $c$ is a facet in the Bruhat–Tits building $\mathcal{B}(G)$). Denote this $G_{c}/G_{c+}$-cuspidal representation by $\rho$. The subgroup $G_{c}/K_{+}$ is a parabolic subgroup of $K/K_{+}$ with Levi factor $G_{c}/G_{c+}$. For simplicity, write this Levi decomposition as $Q = LU$. Where $L$ is the Levi factor and $U$ is its unipotent radical.

There exists a $Q$-embedding of $\rho$ into $W$. Since $\rho$ is $U$-invariant, the image of this embedding lies in $W^{U}$. Note that
$$
W^{U} \subset (V^{K_{+}})^{G_{c+}/K_{+}} = V^{G_{c+}}.
$$
Because $\rho$ is cuspidal, the pair $[\rho,c]$ is an unrefined minimal $K$-type of $V$ by the definition of the unrefined minimal K-type . 

Moreover, we can see that every irreducible component of $V^{K_{+}}$ is a subrepresentation of the parabolic induction from a cuspidal representation of a Levi (this is Harish-Chandra induction since we can find the parabolic and we need not use Deligne-Lusztig induction).
\end{proof}

\begin{lemma}
\label{lem:steinberg-iwahori}
Let $G$ be as in Lemma \ref{lem:cuspidal-support}, let $I$ be its Iwahori subgroup, and let $(\pi,V)$ be an irreducible representation of $G$. If  $V^{K_+}$ contains $\pi_\lambda$ (an irreducible representation lying in the Harish-Chandra principal  series), then $V$ must have a non-zero $I$-fixed vector.
\end{lemma}

\begin{proof}
If $V^{K_+}$ contains $\pi_\lambda$. The cuspidal support of the Steinberg representation is $(T,\mathbf{1})$, where $T$ is a torus and $\mathbf{1}$ denotes the trivial representation. 

By Lemma \ref{lem:cuspidal-support}, there exists a chamber $c$ whose closure contains the vertex associated to $K$ (Assume  the Iwahori subgroup associated to $c$ is $I$ ), such  that Steinberg contains this trivial representation, and the unrefined minimal $K$-type of $V$ is $(c,\mathbf{1})$. This means that the $I/I_{+}$-representation $V^{I_{+}}$ must contain a trivial representation; hence $V^{I} \neq 0$.
\end{proof}

A classical result of Borel and Casselman relates the existence of Iwahori fixed vectors to principal series representations.

\begin{theorem}[Borel–Casselman, \cite{Borel1976}]
\label{thm:borel-casselman}
Let $G$ be a connected reductive group over a non-archimedean local field $k$, and let $I \subset G(k)$ be an Iwahori subgroup. For an irreducible smooth representation $\pi$ of $G(k)$, the following are equivalent:
\begin{enumerate}
    \item $\pi^{I} \neq 0$;
    \item $\pi$ is isomorphic to a subquotient of an unramified principal series representation of $G(k)$.
\end{enumerate}
\end{theorem}

Combining Theorem \ref{thm:borel-casselman} and Lemma \ref{lem:steinberg-iwahori}, we obtain the following immediate corollary.

\begin{corollary}
\label{cor:steinberg-principal-series}
Assume that $G$ satisfies the hypotheses of Lemma \ref{lem:cuspidal-support}. Let $(\pi,V)$ be an irreducible representation of $G$. If $V$ contains $\pi_\lambda$, then $\pi$ must be a subquotient of an unramified principal series.
\end{corollary}

\begin{lemma}
Suppose $G$ is an unramified reductive $p$-adic group, and $(\pi,V)$ is an Iwahori-spherical irreducible representation of $G$. And fix an Iwahori subgroup $I$ which is contained in K. Then every irreducible $K/K_{+}$-subrepresentation $W$ of $V^{K_{+}}$ lies in the principal Harish-Chandra series  $\operatorname{Ind}_{I/K_{+}}^{K/K_{+}}(1)$ (i.e., parabolic induction from the trivial representation of a quasi-split torus).
\end{lemma}
\begin{proof}
For a fixed Iwahori subgroup $I$ of $G$ , By \cite[Theorem 8.4.10(2)]{KalethaPrasad2023}, we know that $I/I_{+}$ is the minimal Levi of $K/K_{+}$. Since every finite Lie group is quasi-split, $I/I_{+}$ is a torus, and $I/K_{+}$ is the Borel containing this torus.

Since $V$ is Iwahori-spherical, $V^I \neq 0$; hence the unrefined minimal $K$-type of $V$ is (the chamber corresponding to $I$, trivial representation). Since two unrefined minimal K type is associated, every unrefined minimal K type of V has the form (chamber, trivial representation).

By the proof of Lemma \ref{lem:cuspidal-support}, we know that for every irreducible component $W$ of $V^{K_{+}}$ , there exists a chamber c (assume the Iwahori subgroup associated to c is I'), such that W  is a subrepresentation of $\operatorname{Ind}_{I'/K_{+}}^{K/K_{+}}(1)$. It is well known that every Borel subgroups of a finite Lie group are rational-conjugated. Hence  $\operatorname{Ind}_{I'/K_{+}}^{K/K_{+}}(1)$ is isomorphic to  $\operatorname{Ind}_{I/K_{+}}^{K/K_{+}}(1)$. Hence the lemma holds.
\end{proof}
\section{Actions of the Hecke  algebras}
Two Hecke algebras occur in the previous sections, so we will compare these two algebras in this section. To simplify terminology, let $H=\operatorname{End}_{K/K_{+}}(\operatorname{Ind}_{I/K_{+}}^{K/K_{+}}(1))$, and let $H(K,I)$ be the finite Hecke algebra of the $p$-adic group (i.e., $I$-bi-invariant functions which are supported on $K$).

\begin{lemma}
There is an isomorphism $\phi$ between $H$ and $H(K,I)$ satisfying $\phi(T_{s_\alpha})=\operatorname{ch}_{Is_{\alpha}I}$ for every simple root $\alpha$.
\end{lemma}
\begin{proof}
This follows directly from the presentations of the two algebras.
\end{proof}

\begin{lemma}
Suppose $G$ is an unramified reductive $p$-adic group with hyperspecial subgroup $K$ and Iwahori subgroup $I$. Let $(\pi,V)$ be an irreducible Iwahori-spherical representation of $G$. As a representation of the $p$-adic group, $V$ has an action coming from $H(K,I)$. As a representation of the finite Lie group $K/K_{+}$, after taking $I/K_{+}$-fixed vectors, $V^I$ is also an $H$-module. Then the two actions coincide.
\end{lemma}
\begin{proof}
Here we emphasize that the measure on $G$ is normalized such that the volume of $I$ is 1. By Lemma A.5, every irreducible component of $V^{K_{+}}$ is a subrepresentation of $\operatorname{Ind}_{I/K_{+}}^{K/K_{+}}(1)$.

In order to show that these two actions coincide, we just need to check the generators $T_{s_\alpha}$ and $\operatorname{ch}_{Is_\alpha I}$. More concretely, we need to show the following equality: for every element $v \in V^I$
$$
T_{s_\alpha} v = \operatorname{ch}_{Is_\alpha I} v.
$$
Assume that the representatives of $Is_{\alpha}I/I$ are $\{u \mid u \in \Lambda\}$.
Then
By the discussion of the action of $T_{s_\alpha}$ , and after applying Frobenius reciprocity, we have:
$$
\operatorname{Hom}_{K/K_{+}}(\operatorname{Ind}_{I/K_{+}}^{K/K_{+}}(1),V^{K_{+}})=\operatorname{Hom}_{I/K_{+}}(1,V^{K_{+}})=V^I
$$
For an element $v \in V^I$, $v$ corresponding to the function $f_v$, where $f_v(gI/K_{+})=gv$, hence 
$$
T_{s_\alpha}v=f_v \circ T_{s_\alpha}(I/K_{+})
$$
$$
 = f_v\bigl(\sum_{u \in  \Lambda}uI/K_{+}\bigr).
$$
$$
= \sum_{u \in \Lambda} u v,
$$
$$
\text{RHS} = \int_{Is_{\alpha}I} x v \, dx = \sum_{u\in \Lambda} u v.
$$
Hence LHS = RHS.
\end{proof}

\end{document}